\documentclass[12pt]{article}
\usepackage{cite}
\usepackage{mathrsfs, amsmath,amssymb,amsfonts, amsthm, dsfont, tikz}
\newtheorem{theorem}{Theorem}[section]
\newtheorem{proposition}[theorem]{Proposition}

\newtheorem{lemma}[theorem]{Lemma}

\newtheorem{definition}[theorem]{Definition}

\newcommand*\circled[1]{\tikz[baseline=(char.base)]{
            \node[shape=circle,draw,inner sep=1pt] (char) {#1};}}
\begin{document}
\title{On sparsity of positive-definite automorphic forms within a family}
\author{Junehyuk Jung}
\date{\today}
\maketitle
\begin{abstract}
In \cite{bm}, Baker and Montgomery prove that almost all Fekete polynomials under certain ordering have at least one zero on the interval $(0,1)$. In terms of the positve-definiteness, Fekete polynomial has no zero on the interval $(0,1)$ if and only if the corresponding automorphic form is positive-definite. On generalizing \cite{bm}, we formulate an axiomatic result about sets of automorphic forms $\pi$ satisfying certain averages when suitably ordered, which ensures that almost all $\pi$'s are not positive-definite within such sets. We then apply the result to various families, including the family of holomorphic cusp forms, the family of the Hilbert class characters of imaginary quadratic fields, and the family of elliptic curves.
\end{abstract}
\section{Introduction}
Let $\pi$ be a self-dual automorphic form on $GL_m/\mathbb{Q}$ and let $\Lambda(s,\pi)$ be its completed standard $L$-function. Put 
\[
\Lambda_0(s,\pi) = \Lambda(s,\pi) (s(1-s)/2)^k
\]
where $k$ is the order of pole of $\Lambda(s,\pi)$ at $s=1$. We say $\pi$ (or equivalently, $\Lambda(s,\pi)$) is \textit{positive-definite} if $\Lambda_0(\frac{1}{2}+it,\pi)$ is a positive-definite function of $t$ in the additive group $\mathbb{R}$. The interest in such $\pi$'s is that $\Lambda(s,\pi)$ has no real zeros (hence no Siegel zero). There are many positive-definite $\pi$'s starting with $\pi = \bold{1}$, i.e. the Riemann zeta function, and most $\pi$'s of small conductor are positive-definite. The central question is whether (for a given $m$) the set of positive-definite $\pi$'s is finite or not. See \cite{eitan} for a further discussion of this question and for some finiteness results for $\pi$'s lying in certain families. 

Our aim in this paper is to show that within a given family $\mathscr{F}$ of $\pi$'s, positive-definiteness is sparse - that is, almost all of the members are not positive-definite. For a definition of a family of $\pi$'s see \cite{fa} and \cite{kowal}. We have not dealt with the most general family as defined there and so we will not repeat the definition here. Instead our main results  prove this sparsity for various interesting families which have different flavors and which are indicative of the general phenomenon. 

To that end we formulate an axiomatic result about sets of automorphic forms $\pi$ satisfying certain averages when suitably ordered and which ensures that almost all $\pi$'s are not positive-definite.

Let $(\mathscr{F}, \mathscr{N})$ be a pair with
\begin{itemize}
\item $\mathscr{F}$ is a set of automorphic forms on $GL_m/\mathbb{Q}$ for some fixed $m\geq 1$, and
\item $\mathscr{N}:\mathscr{F} \to \mathbb{N}$ is an ordering of $\mathscr{F}$ such that
\[
S(X) = \{\pi \in \mathscr{F}| \mathscr{N}(\pi)<X\}
\]
is a finite set for any $X>0$.
\item Each $\pi \in \mathscr{F}$ is cuspidal and self-dual.
\item $\gamma(s)=L(s,\pi_\infty)$ is the same for every $\pi$, and we assume $\pi_\infty$ is tempered (so that $\gamma(s)$ has no pole in $\mathrm{Re}(s)>0$.)
\end{itemize}

We say $(\mathscr{F}, \mathscr{N})$ is statistically balanced and fluctuating if the assumption $\circled{A}$ and $\circled{B}$ concerning the averages of the coefficients of the $\pi$'s when ordered by $\mathscr{N}$ in Proposition \ref{def} are satisfied. Our main technical device is the following:
\begin{lemma}\label{theorem}
If $(\mathscr{F}, \mathscr{N})$ is statistically balanced and fluctuating then almost all $\pi\in\mathscr{F}$ are not positive-definite as $X \to \infty$.
\end{lemma}
Here and elsewhere, we say \textit{almost all} when the corresponding set $A \subset \mathscr{F}$ satisfies 
\[
\lim_{X\to \infty} \frac{|A \cap S(X)|}{|S(X)|} =1.
\]

Note that if $m=1$ and $\mathscr{F}$ is the (universal) family of real Dirichlet characters ordered by conductor, then the above is a result of Baker and Montgomery \cite{bm}. Their approach is via the logarithmic derivative of $L(s,\chi)$ and examining the behavior on $(1/2,1)$ and for $s$ near $1/2$. Among their inputs are strong density theorem for the location of the zeros of these functions. For the more general families that we study, such density theorem are not known. We therefore give a more direct treatment of Lemma \ref{theorem} avoiding these density theorems, but still using a number of the probabilistic ideas from \cite{bm}.

Our main results are proven by separately proving various families $\mathscr{F}$ of $\pi$'s are statistically balanced and fluctuating. These are achieved by spectral techniques (trace formulae) and arithmetic geometric techniques (monodromy). We state these as follows:

\begin{theorem}\label{cor1}
Almost all members in the family $\mathscr{F}$ are not positive-definite for the following cases:
\begin{enumerate}
\item $m=2$; $\mathscr{F} = \bigcup_{q\text{:squarefree}} \mathscr{F}(k,q)$, where $\mathscr{F}(k,q)$ is the set of primitive holomorphic cusp forms of weight $k$ on level $q$, as $q \to \infty$.
\item $m=2$; $\mathscr{F} = \bigcup_{D} \mathscr{F}(D)$, where $\mathscr{F}(D)$ is the set of holomorphic modular forms of weight $1$ which are associated to a Hilbert class character of an imaginary quadratic field of discriminant $-D$, as $D \to \infty$.
\item $m=3,4,5$; $\mathscr{F}_{m-1} = \bigcup_{q\text{:squarefree}} \mathscr{F}_{m-1}(k,q)$, where $\mathscr{F}_{m-1}(k,q)$ is the set of $L$-functions of the form $L(s,\pi,\text{sym}^{m-1})$, where $\pi$ is a primitive holomorphic cusp form of weight $k$ on level $q$ (these are Euler products of degree $m$ and correspond to automorphic forms on $GL_m$ over $\mathbb{Q}$ \cite{sh}), as $q \to \infty$.
\end{enumerate}
\end{theorem}
\paragraph{Note} All the automorphic forms in the theorem are self-dual, as they have either trivial or quadratic central character.

Next we consider families of elliptic curves.
\begin{theorem}\label{cor3}
Almost all $L$-functions $L(s,E)$ of the elliptic curves $E$ over $\mathbb{Q}$, when ordered by naive height, are not positive-definite.
\end{theorem}
\begin{theorem}\label{cor4}
Let $\{E(t)\}_{t \in \mathbb{\mathbb{Q}}}$ be the one-parameter family of elliptic curves given by some fixed polynomials $a$ and $b$ such that
\[
E(t): \text{ } y^2 = x^3 - a(t)x+b(t).
\]
where the $j$-invariant is non-constant. If we order $\{E(t)\}_{t \in \mathbb{Q}}$ by the height of $t$: 
\[
\mathscr{N}(t)=\max\{|n|,|m|:(n,m)=1,t=\frac{n}{m}\},
\]
then almost all $L$-functions attached to $E(t)$ are not positive-definite.
\end{theorem}

\paragraph{Acknowledgements:} I would like to thank Peter Sarnak for introducing the problem and encouragement. I also thank Sug Woo Shin for several valuable discussions. I am grateful to Hyungwon Kim for helping me with the numerical experiments.

\section{Preliminaries}
\subsection{Asymptotic analysis and distribution of $\pi_p$}
We first introduce a notation that will appear frequently in our work.
\begin{definition}\label{not}
Let $(\mathscr{F},\mathscr{N})$ be a family of automorphic forms. Put
\[
S(X) = \{\pi \in \mathscr{F}| \mathscr{N}(\pi)<X\}.
\]
For any mapping $F:\mathscr{F} \to \mathbb{C}$, we define
\[
E_{(\mathscr{F},\mathscr{N})}(F(\pi)) = \lim_{X \to \infty}\frac{1}{|S(X)|}\sum_{\pi \in S(X)} F(\pi),
\]
if the limit exists. We will write $E(F(\pi))$ instead of $E_{(\mathscr{F},\mathscr{N})}(F(\pi))$ whenever there is no confusion. 
\end{definition}
Under this notation, the asymptotic density of $A \subseteq \mathscr{F}$ is given by $E_{(\mathscr{F},\mathscr{N})}(I_A(\pi))$ where
\begin{align*}
I_A(\pi)&=1 \text{ (if } \pi \in A \text{)}\\
&=0 \text{ (if } \pi \notin A \text{)}.
\end{align*}
Let the local $L$-function attached to $\pi$ on $GL_m/\mathbb{Q}$ at a finite prime $p$ be given by
\[
L(s,\pi_p) = \prod_{j=1}^m (1-\frac{\alpha_j(\pi_p)}{p^s})^{-1}
\]
for some complex numbers $\alpha_1(\pi_p) , \cdots , \alpha_m(\pi_p)$ and let the Dirichlet series for $L(s,\pi)$ be given by
\[
L(s,\pi) = \prod_{p\text{: finite prime}}L(s,\pi_p)= \sum_{n=1}^\infty \frac{\lambda_\pi(n)}{n^s}.
\]
Put 
\[
\Psi(\pi_p):=(\alpha_1(\pi_p),\cdots,\alpha_m(\pi_p)).
\]
Let $\overline{\mathbb{D}} \subset \mathbb{C}$ be the set of complex numbers of modulus less than or equal to $1$. Now we define some conditions on the family.
\begin{proposition}\label{def}
For any given $(\mathscr{F},\mathscr{N})$, we define $\circled{A}$ and $\circled{B}$ as:
\begin{itemize}
\item[$\circled{A}$:] There exists a permutation invariant measure $\mu_p$ on $\overline{\mathbb{D}}^m$ for each prime $p$ such that  $|S|$-tuple $\big(\Psi(\pi_p)\big)_{p \in S}$ is equidistributed with respect to $\prod_{p \in S} \mu_p$ for any finite set of primes $S$. 
\item[$\circled{B}$:] Following estimations concerning the asymptotic distribution of $\lambda_\pi(p)$ are satisfied:
\begin{align*}
E_{(\mathscr{F},\mathscr{N})}(\lambda_\pi(p)^2) & \gg 1\\
E_{(\mathscr{F},\mathscr{N})}(\lambda_\pi(p)) &\ll p^{-\frac{1}{2}}.
\end{align*}
\end{itemize}
\end{proposition}
\subsection{Probabilistic theory}
In this section we list lemmas without proofs that will be used in subsequent chapters. We refer the reader to \cite{bm} for the proofs.
\begin{lemma}\label{lemma4}
Suppose that for $r=1,2,3,\cdots$ the random variables $Z_{rn}$ are independent, where $1\leq n \leq N_r$, and put
\[
Z_r=\sum_{n=1}^{N_r} Z_{rn}.
\]
Suppose that $E(Z_{rn})=0$ for all $n$ and $r$, and that $P(|Z_{rn}|\leq c_n)=1$ for all $n$ and $r$, where $c_n\geq 0$ are constants such that
\[
\sum_{n=1}^\infty c_n^3 <\infty.
\]
Let
\[
\sigma(r)= \big(\sum_{n=1}^{N_r} Var(Z_{rn})\big)^{1/2}
\]
denote the standard deviation of $Z_r$, and suppose that $\sigma(r)\to \infty$ as $r \to \infty$. Then the distribution of the random variable $Z_r/\sigma(r)$ tends to the normal distribution with $\mu=0$ and $\sigma=1$ as $r \to \infty$.
\end{lemma}
\begin{definition}
For a sequence of real numbers $a_1, \cdots, a_n$,
\[
S^-(a_1, \cdots, a_n)
\]
is the number of sign changes in the sequence with zero terms deleted, and
\[
S^+(a_1, \cdots, a_n)
\]
is the maximum number of sign changes with zero terms replaced by number of arbitrary sign. If $f$ is a real-valued function defined on an interval $(a,b)$, then $S^\pm (f;a,b)$ denotes the supremum of 
\[
S^\pm (f(a_1), \cdots, f(a_n))
\] 
over all finite sequences for which $a<a_1<\cdots<a_n<b$.
\end{definition}
\begin{lemma}\label{lemma5}
Let $\delta>0$ and suppose that $Z_1, \cdots, Z_R$ are independent random variables such that $P(Z_r>0)\geq \delta$ and $P(Z_r<0)\geq \delta$ for all $r$. Then
\[
P(S^- (Z_1, \cdots, Z_R)\leq \frac{1}{5} \delta R) \ll e^{-\delta R/3} 
\]
uniformly in $\delta$ and $R$.
\end{lemma}
Here and elsewhere, we write $A \ll_\tau B$ to mean $|A| \leq C(\tau)B$ for some constant $C(\tau)$ depending only on $\tau$. 
\begin{lemma}\label{lemma6}
Let $f$ be a real-valued function defined on $\mathbb{R}$ which is Riemann integrable on finite intervals, and suppose that the Laplace transform
\[
\mathscr{L}(s) =\int_{-\infty}^\infty f(x) e^{-sx}dx
\]
converges for all $s>0$. Then
\[
S^-(f;-\infty , +\infty) \geq S^+(\mathscr{L}; 0,\infty).
\]
\end{lemma}
\section{Approximation}
For Section 3 and 4, we assume that the family $(\mathscr{F},\mathscr{N})$ satisfies $\circled{A}$ and $\circled{B}$. 
\subsection{Mellin transform of an automorphic $L$-function}
Let
\[
\phi_\pi(y)=\frac{1}{2\pi i} \int_{(2)} \Lambda(s,\pi) y^{-s} ds
\]
and
\[
W(y)=\frac{1}{2\pi i} \int_{(2)} \gamma(s) y^{-s} ds
\]
where $(\sigma)$ denotes the contour given by $\mathrm{Re}(s)=\sigma$ pointing upward. Because $L(s,\pi)$ is of finite order and $\gamma(s)$ is a product of gamma functions with no pole on $\mathrm{Re}(s)>0$, both integrations converge absolutely for all $y>0$. By shifting contour to the right, we see that $W(y)$ is rapidly decreasing, and by shifting contour to $-\epsilon$ for small enough $\epsilon>0$, we see that $W(y)$ is bounded. In particular from the bound $|\alpha_j (\pi_p)| \leq \sqrt{p}$
\[
\phi_\pi(y)=\sum_{n=1}^\infty \lambda_\pi(n) W(ny)
\]
converges absolutely for all $y>0$ and is rapidly decreasing in $y$. Now applying the Mellin inversion formula to the functional equation of $\Lambda(s,\pi)$, we get
\[
\phi_\pi(y)=\epsilon_\pi \phi(\frac{1}{N_\pi y}),
\]
hence $\phi_\pi(y) \ll_k y^k$ for all $k>0$ as $y\to 0$. This implies that we have
\[
\Lambda(s,\pi)=\int_0^\infty \phi_\pi(y)y^{s-1}dy
\]
for all $s \in \mathbb{C}$ and that the integration converges absolutely. 

We write $\phi_\pi^N(y)$ for 
\[
\phi_\pi^N(y) = {\sum}^N \lambda_\pi(n) W(ny)
\]
where ${\sum}^N$ is defined by $\sum_{p|n\Rightarrow p<N}$. Define
\[
L^N(\pi,s)=\prod_{p<N} \prod_{j=1}^m (1-\alpha_j(\pi_p)p^{-s})^{-1} = {\sum}^N \frac{\lambda_\pi(n)}{n^s}
\]
and let
\[
\Lambda^N(\pi,s)=L^N(\pi,s)L(\pi_\infty,s).
\]
Then it follows that
\[
\phi_\pi^N(y)= \frac{1}{2\pi i }\int_{(2)}\Lambda^N(\pi,s)y^{-s} ds.
\]
\begin{lemma}\label{lemma1}
There exists a constant $A>0$ such that
\[
W(y) \ll e^{-m\pi y^{2/m}} y^A
\]
for $y>1$.
\end{lemma}
\begin{proof}
Let
\[
\gamma(s)=\pi^{-ms/2}\prod_{j=1}^m \Gamma(\frac{s+\kappa_j}{2})
\]
and
\[
\frac{1}{m}\sum_{j=1}^m \kappa_j-1 = \kappa \in \mathbb{R}.
\]
By Stirling's formula \cite{le}, for $\mathrm{Re}(s)>1$,
\begin{align*}
\prod_{j=1}^m \Gamma(\frac{s+\kappa_j}{2}) & \ll |\exp\big(\sum_{j=1}^m \frac{s+\kappa_j-1}{2}\log \frac{s}{2} - \frac{ms}{2}\big)|\\
&=|\exp\big( \frac{m(s+\kappa)}{2}\log \frac{s}{2}-\frac{ms}{2}\big)|.
\end{align*}
Put $\mathrm{Re}(s)=\sigma$, $\mathrm{Im}(s)=t$, and $\tau=\frac{t}{\sigma}$. Define $\psi(x)=\log\sqrt{1+x^2}-x\arctan x$.
\begin{align*}
&=\exp\big(\frac{m(\sigma+\kappa)}{2}\log \frac{\sigma\sqrt{1+\tau^2}}{2}-\frac{m\sigma\tau}{2}\arctan \tau -\frac{m\sigma}{2} \big)\\
&=\exp\big(\frac{m(\sigma+\kappa)}{2}\log \frac{\sigma}{2} -\frac{m\sigma}{2} +\frac{m\kappa}{2}\log \sqrt{1+\tau^2}+\frac{m\sigma}{2} \psi(\tau) \big)
\end{align*}
Therefore
\begin{align*}
W(y) \ll& (y\pi^{m/2})^{-\sigma} \exp\big(\frac{m(\sigma+\kappa)}{2}\log \frac{\sigma}{2} -\frac{m\sigma}{2}\big) \\
&\times \int_{-\infty}^\infty \sigma (1+\tau^2)^{m\kappa/4}\exp\big(\frac{m\sigma}{2} \psi(\tau) \big) d\tau\\
\ll &(y\pi^{m/2})^{-\sigma} \big( (\frac{\sigma}{2})^{m/2}\big)^\sigma e^{-m\sigma/2} \sigma^{(1+m\kappa)/2} \\
\ll& e^{-m\pi y^{2/m}} y^{\kappa+1/m}
\end{align*}
where we put $\sigma=2\pi y^{2/m}$ on the last inequality assuming $y>1$.
\end{proof}
\begin{lemma}\label{lemma2}
For any positive $\epsilon>0$,
\[
\phi_\pi^N(y) \ll_{N,\epsilon} y^{-(1/2-1/(m^2+1))-\epsilon}.
\]
\end{lemma}
\begin{proof}
Since
\[
\phi_\pi^N(y) = \frac{1}{2\pi i} \int_{(2)} L(s,\pi_\infty)\prod_{p<N}L(s,\pi_p) y^{-s} ds,
\]
and $L(s,\pi_p)$ is holomorphic on $\mathrm{Re}(s)>1/2-1/(m^2+1)$ for all $p$ \cite{grc}, we can shift contour to $(\sigma)$ for any $\sigma>1/2-1/(m^2+1)$, to obtain the desired estimation.
\end{proof}
\subsection{Approximation}
In this section, we fix $N>0$ large and study how $\phi_\pi(y)$ is well approximated by $\phi_\pi^N(y)$. Firstly, note that the first $N$ terms of $\phi_\pi^N(y)$ and $\phi_\pi(y)$ agree. Hence, on range $y>N^{-1/2}$, $\phi_\pi(y)-\phi_\pi^N(y)$ will be negligible. We quantify this as follows:
\begin{lemma}\label{lemma8}
Assume $1/2<\mathrm{Re}(s)=\sigma<5/4$
\[
\int_{1/\sqrt{N}}^\infty \big(\phi_\pi (y) - \phi_\pi^N (y)\big)y^{s-1} \log y dy =O(e^{-mN^{1/m}})
\]
\end{lemma}
\begin{proof}
\begin{align*}
\int_{1/\sqrt{N}}^\infty \big( \phi_\pi (y) - &\phi_\pi^N (y) \big) y^{s-1} \log y dy \ll \int_{1/\sqrt{N}}^\infty \sum_{n>N} \big{|}n W(ny) y^{\sigma-1} \log y\big{|} dy\\
& \ll \int_{1/\sqrt{N}}^\infty \sum_{n>N} n (ny)^A \exp\big(-m\pi (ny)^{2/m}\big) (y^{1/4}+y^2) dy\\
& \ll \int_{1/\sqrt{N}}^\infty \int_N^\infty x (xy)^A \exp\big(-m\pi (xy)^{2/m}\big) (y^{1/4}+y^2) dxdy\\
& \ll \int_{N^{-1/m}}^\infty \int_{N^{2/m}}^\infty y \exp\big(-m xy\big) dxdy\\
& \ll e^{-mN^{1/m}}
\end{align*}
Here we used Lemma \ref{lemma1} on the second inequality.
\end{proof}
Since we are fixing $N>0$, we can treat $\phi_\pi^N(y)$ as a random series constructed from finitely many random variables, by varying $\pi$ over the family.
\begin{lemma}\label{lemma7}
Assume $\mathrm{Re}(s)\geq \frac{1}{2}+\epsilon_0$. For some $D>0$ depending only on the family, we have
\[
E\big(|\int_0^{1/\sqrt{N}} \phi_\pi^N(y) y^{s-1} \log y dy|\big) \ll \frac{1}{N^{\epsilon_0/4} \epsilon_0^D}
\]
uniformly in $s$.
\end{lemma}
\begin{proof}
From $\circled{A}$ and $\circled{B}$, there exists a sufficiently large integer $k$ such that
\[
E(\lambda_\pi(n_1)\lambda_\pi(n_2)) < \frac{d_k(n_1 n_2)}{\sqrt{R(n_1n_2)}}
\]
for any fixed pair of positive integers $n_1$ and $n_2$ where $R(n)=\prod_{p||n}p$.

Here $d_k(n)$ is the $k$-th divisor function defined by
\[
\zeta(s)^k=\sum_{n=1}^\infty \frac{d_k(n)}{n^s}
\] 
for $\mathrm{Re}(s)>1$.

Therefore we have
\begin{align*}
E\big(|\phi_\pi^N(y)|^2\big) &< \sum_{n_1,n_2\geq 1} \frac{d_k(n_1 n_2)}{\sqrt{R(n_1n_2)}} |W(n_1 y) W(n_2 y)|\\
&=\sum_{n_1,n_2<Y} + \sum_{n_1 >Y~or~n_2>Y},
\end{align*}
where $Y$ is an auxiliary variable to be chosen later. 
\begin{align*}
\sum_{n_1,n_2<Y}&\ll \sum_{n_1,n_2<Y} \frac{d_k(n_1n_2)}{\sqrt{R(n_1n_2)}}\\
&\ll \sum_{n<Y^2}\frac{d_k(n)d_1(n)}{\sqrt{R(n)}}
\end{align*}
Note that
\[
\sum_{n=1}^\infty \frac{d_k(n)d_1(n)}{\sqrt{R(n)}}n^{-s} = \prod_p (1+\frac{2d_{k+1}(p)}{p^{\frac{1}{2}+s}}+\frac{3d_{k+1}(p^2)}{p^{2s}}+\frac{4d_{k+1}(p^3)}{p^{3s}}+\cdots)
\]
and therefore 
\[
\sum_{n_1,n_2<Y} \ll Y \log^B Y
\]
for some $B>0$, by standard Tauberian theorems \cite{mv}. 

Now we treat the second summation using Lemma \ref{lemma1} as follows:
\begin{align*}
\sum_{n_1 >Y~or~n_2>Y} &\ll \sum_{n_1\geq 1}^\infty n_1 W(n_1 y)\sum_{n_2>Y} n_2 W(n_2 y)\\
&\ll \int_0^\infty te^{-2(ty)^{2/m}} dt\int_Y^\infty te^{-2(ty)^{2/m}} dt\\
&\ll y^{-4} \int_{yY}^\infty te^{-2t^{2/m}} dt\\
&\ll y^{-4} e^{-(yY)^{2/m}}.
\end{align*}
Choosing $Y=\frac{1}{y} (4 \log \frac{1}{y})^{\frac{m}{2}}$, we get
\[
E\big(|\phi_\pi^N(y)|^2\big) \ll \frac{1}{y} \log^{B+\frac{m}{2}} \frac{1}{y}.
\]
Therefore 
\begin{align*}
|E\big(\int_0^{1/\sqrt{N}} \phi_\pi^N(y)y^{s-1} \log y dy\big)| &\leq \int_0^{1/\sqrt{N}}E\big( | \phi_\pi^N(y)|\big) y^{\sigma-1} \log \frac{1}{y}dy\\
&\leq  \int_0^{1/\sqrt{N}}E\big(|\phi_\pi^N(y)|^2\big)^{1/2} y^{\sigma-1} \log \frac{1}{y}dy\\
&\ll \int_0^{1/\sqrt{N}}y^{\sigma-3/2} \log^{\frac{2B+m+1}{4}} \frac{1}{y}dy
\end{align*}
where interchanging the integration and $E(\cdot)$ is justified by Lemma \ref{lemma2} and the dominant convergence theorem. Finally, let $\frac{2B+m+1}{4} =D-1$, to get
\begin{align*}
&=\int_0^{1/\sqrt{N}}y^{\sigma-3/2} \log^{D-1} \frac{1}{y}dy\\
&\ll  \frac{1}{\epsilon_0^{D-1}}\int_0^{1/\sqrt{N}}y^{\sigma-3/2-\epsilon_0/2}  dy\\
&\ll \frac{1}{\epsilon_0^{D-1}}\int_0^{1/\sqrt{N}}y^{-1+\epsilon_0/2}  dy\\
&\ll \frac{1}{N^{\epsilon_0/4} \epsilon_0^D}.
\end{align*}
\end{proof}

\section{Sparsity of the positive-definite forms}
In this section, for simplicity, we further assume that $E_{(\mathscr{F},\mathscr{N})}(\lambda_\pi(p))=0$ for all prime $p$. After establishing the theory, modification of the proof to remove this assumption is straightforward.
\subsection{Oscillation of $-\frac{{L^N}'}{L^N}(s,\pi)$}
From $\mu_p$ for which $\Psi(\pi_p)$ is equidistributed, we can find the limiting distribution $X_p$ of $\lambda_\pi (p)$ which is supported on $[-m,m]$. From asymptotic independence of $\{\Psi(\pi_p)\}_{p \in S}$, we deduce that $\{\lambda_\pi(p)\}_{p \in S}$ are asymptotically independent for any finite set of primes $S$.

For $R>0$, put $R_1=[\frac{R}{5}]$, and for $R_1<r \leq R$, define
\begin{align*}
s_r&=\frac{1}{2}+\exp(-4^r)\\
u(s)&=\exp((s-\frac{1}{2})^{-1/2})\\
v(s)&=\exp((s-\frac{1}{2})^{-2}).
\end{align*}
\begin{lemma}\label{lemma9}
Let $C>0$ be a fixed constant. Putting $N=v(s_R)$, there exist $R_1<r_1(\pi)<r_2(\pi)<r_3(\pi)<R$ such that
\[
\frac{C}{2s_{r_i(\pi )}-1} < (-1)^i\sum_{p<N} \frac{\lambda_\pi(p) \log p}{ p^{s_{r_i(\pi)}}}
\]
for all but $O(1/R)$ of $\pi$'s.
\end{lemma}
\begin{proof}
Assume $\frac{1}{2}<s$ and $v(s)<N$. We split
\[
\sum_{p<N} \frac{\lambda_\pi(p) \log p}{p^s}
\]
into
\[
I_1(\pi,s)=\sum_{p\leq u(s)}, \text{ } I_2(\pi,s)=\sum_{u(s) \leq p<v(s)}, \text{ } I_3(\pi,s)=\sum_{v(s)<p<N}.
\]
Then
\[
E(I_1(\pi,s)^2) \ll \sum_{p\leq u(s)}\frac{(\log p)^2}{p} \ll (\log u(s))^2 \ll \frac{1}{2s-1}
\]
and
\begin{align*}
E(I_3(\pi,s)^2) &\ll \sum_{p>v(s)} \frac{(\log p)^2}{p^{2s}}\\
&\ll (2s-1)^{-1}v(s)^{1-2s}(\log v(s))^2\\
&\ll (2s-1)^{-3} \exp(-2(s-1/2)^{-1})\\
&\ll \exp(-(s-1/2)^{-1}) \ll 2s-1
\end{align*}
Therefore
\[
|I_1(\pi,s_r)|+|I_3(\pi,s_r)|<\frac{1}{2s_r-1}
\]
holds for all $R_1 <r\leq R$ for all but $O(e^{-R})$ of $\pi$.

Now we consider $I_2$. This has the asymptotic distribution which is the same as the distribution function of
\[
X(s)=\sum_{u(s) \leq p<v(s)}\frac{\log p}{p^s} X_p.
\]
Let $\rho(s)$ be the standard deviation of $X(s)$. Then by Lemma \ref{lemma4},
\[
X(s)/\rho(s)
\]
converges to the normal distribution $N(0,1)$. By the assumption $\circled{B}$,
\[
\rho(s) > c (\sum_{u(s) \leq p<v(s)}\frac{(\log p)^2}{p^{2s}})^{\frac{1}{2}} \sim \frac{c}{2s-1}
\]
for some constant $c>0$. Let $\Phi(x)$ be the cumulative normal distribution function. Pick $\delta$ so that $0<\delta<\Phi(-\frac{C+1}{c})$. Then assuming $R$ large enough,
\[
P(X(s_r)>\frac{C+1}{2s_r-1}) \geq \delta, \text{  } P(X(s_r)<-\frac{C+1}{2s_r-1}) \geq \delta,
\]
for each $R_1<r\leq R$. Define $B_r$ as follows:
\begin{equation*}
B_r=\left \{
\begin{array}{rl}
1 & \text{if } X(s_r)>\frac{C+1}{2s_r-1},\\
-1 & \text{if } X(s_r)<-\frac{C+1}{2s_r-1},\\
0 & \text{otherwise}.
\end{array} \right.
\end{equation*}
Since the intervals $(u(s_r),v(s_r)]$ are disjoint, the variables $Z(s_r)$ are independent. Hence by Lemma \ref{lemma5},
\[
P(S^- (B_{R_1 +1}, B_{R_1 +2}, \cdots, B_R)\leq \delta(R-R_1)/5) \ll \exp(-\delta(R-R_1)/5).
\]
Therefore after taking $R$ sufficiently large, together with the estimation on $|I_1|+|I_2|$, we get the assertion.
\end{proof}
Note that
\[
-\frac{{L^N}'}{L^N}(\pi,s)=\sum_{p<N} \sum_{k=1}^\infty \frac{\Lambda_\pi(p^k)}{p^{ks}}
\]
where 
\[
\Lambda_\pi(p^k)=\log p \sum_{j=1}^m \alpha_j(\pi_p)^k
\]
By the assumption $\circled{A}$, for almost all $\pi$, contribution from $k>2$ is $O(1)$. Also by the same assumption, contribution from $k=2$ is $O(\frac{1}{2s-1})$.

Therefore taking $C$ large enough in Lemma \ref{lemma9}, we deduce:
\begin{lemma}\label{lemma10}
Putting $N=v(s_R)$, there exist $R_1<r_1(\pi)<r_2(\pi)<r_3(\pi)<R$ such that
\[
\frac{1}{2s_{r_i(\pi)}-1} < (-1)^i\frac{{L^N}'}{L^N}(\pi,s_{r_i(\pi)})
\]
for all but $O(1/R)$ of $\pi$'s.
\end{lemma}
\subsection{Proof of Lemma \ref{theorem}}
Put $N=v(s_R)$ and $\epsilon_0=s_R-\frac{1}{2}$ and assume $N$ is large. Then
\begin{align*}
\epsilon_0&=\frac{1}{\sqrt{\log N}}\\
N^{\epsilon_0}&=\exp(\sqrt{\log N}).
\end{align*}
By Lemma \ref{lemma8} and \ref{lemma7}, for $s \in [\frac{1}{2}+\epsilon_0,1]$
\[
|\int_0^\infty \phi_\pi^N(y)y^{\sigma-1} \log y dy - \int_{1/\sqrt{N}}^\infty \phi_\pi(y)y^{\sigma-1} \log y dy|< \exp(-\sqrt{\log N}/8)
\]
for all but $O((\log N)^{2D}\exp(-\sqrt{\log N}/8))$ forms.
Note that
\begin{align*}
\int_0^\infty \phi_\pi^N(y)y^{s-1} dy &= \Lambda^N(\pi,s)\\
&=\gamma(s)\prod_{p<N}L(\pi_p,s),
\end{align*}
hence
\[
\int_0^\infty \phi_\pi^N(y)y^{s-1} \log y dy = \gamma(s)\prod_{p<N}L(\pi_p,s)\big(\frac{\gamma'(s)}{\gamma(s)}+\frac{{L^N}'}{L^N}(\pi,s)\big).
\]
By the assumption $\circled{B}$ and the Mertens' third theorem
\[
E(\prod_{p<N}L(\pi_p,s)^{-1}) \ll \prod_{p<N} (1+\frac{1}{p}) \ll \log N.
\]
Therefore
\[
\prod_{p<N}L(\pi_p,s)>\frac{1}{(\log N)^2}
\]
for all but $O(\frac{1}{\log N})$ of $\pi$'s. 

Observe that $\gamma(s)$ is bounded away from $0$ and $\gamma'(s)/\gamma(s)=O(1)$ for $s \in (1/2,1)$. Hence by Lemma \ref{lemma10}, we find $r_1(\pi)<r_2(\pi)<r_3(\pi)$ so that
\[
\frac{1}{(\log N)^2}<(-1)^i \int_0^\infty \phi_\pi^N(y)y^{s_{r_i(\pi)}-1} \log y dy 
\]
for all but $O(\frac{1}{\log N}+\frac{1}{R})$ of $\pi$'s. Combining altogether, we see that
\[
\int_{1/\sqrt{N}}^\infty \phi_\pi(y)y^{s-1}\ln y dy
\]
has at least two sign changes in $s \in (1/2,1)$ for all but $O(1/R)$ forms. Therefore, by Lemma \ref{lemma6}, except $O(1/R)$ of $\pi$, $\phi_\pi(y)$ has at least one sign change on $(0,\infty)$. Since $R$ can be chosen arbitrarily large, we conclude that almost all $\pi$ are not positive-definite.

\section{Result I}
In this section, we prove Theorem \ref{cor1} by verifying $\circled{A}$ and $\circled{B}$ for each given family.
\subsection{Holomorphic modular form}
Let $\mathscr{F} = \bigcup_{q\text{:squarefree}} \mathscr{F}(k,q)$, where $\mathscr{F}(k,q)$ is the set of primitive holomorphic cusp forms of weight $k$ on level $q$, and let $\mathscr{N}(\pi)$ be the level of $\pi$. Note that for a primitive holomorphic cusp form $\pi$ of trivial nebentypus, $\lambda_\pi (n) \in \mathbb{R}$ for all $n \in \mathbb{N}$ and 
\begin{equation}\label{observe}
\{|\alpha_1(\pi_p)|, |\alpha_2(\pi_p)|\} = \left\{
\begin{array}{rl} \{1,1\} & \text{if } p \nmid cond(\pi),\\ \{\frac{1}{\sqrt{p}},0\} & \text{if }p|| \text{ }cond(\pi).
\end{array} \right.
\end{equation}
Therefore from \cite{se} and \cite{ham}, we see that $(\mathscr{F},\mathscr{N})$ satisfies $\circled{A}$, assuming the existence of a constant $\alpha_p \in [0,1]$ such that
\[
\alpha_p = \lim_{X \to \infty} \frac{\sum_{q<X, p \nmid q}|\mathscr{F}(k,q)|}{\sum_{q<X}|\mathscr{F}(k,q)|}.
\]
If this is the case, the asymptotic distribution of $\lambda_\pi(p)$ is given by
\[
X_p=\alpha_p \frac{p+1}{\pi}\frac{\sqrt{1-x^2/4}}{(p^{1/2}+p^{-1/2})^2-x^2}dx + (1-\alpha_p)(\frac{1}{2} \delta_{-1/\sqrt{p}}+\frac{1}{2} \delta_{1/\sqrt{p}})
\]
where $\delta_a$ is the Dirac delta measure concentrated at $a$. 

In order to prove the existence of $\alpha_p$, we explicitly compute it by using the dimension formula for the space of the newforms \cite{dim}, yielding $\alpha_p=1/(1+1/p-1/p^2)$. From this, $E(\lambda_\pi(p)^2)=1+O(1/p)$ and $E(\lambda_\pi (p))=0$, and therefore $\circled{B}$ also holds.

\subsection{Dihedral forms}
For a squarefree integer $D>3$ with $D \equiv 3(\mathrm{mod } 4)$ let $\psi$ be a character of the ideal class group of the imaginary quadratic field $\mathbb{Q}(\sqrt{-D})$ which is not a genus character. For such $\psi$, one can associate a primitive holomorphic cusp form in $S_1(\Gamma_0(D),\chi_{-D})$ where $\chi_{-D}$ is the unique primitive quadratic character modulo $D$. Let $\mathscr{F}(D)$ be the set of all such forms. Note that the size of $\mathscr{F}(D)$ is given by
\[
|\mathscr{F}(D)|=h_{-D} - 2^{\omega(D)} 
\]
where $\omega(n)$ is the number of distinct prime divisors of $n$. By the theorem due to Siegel, $h_{-D} \gg_\epsilon D^{1/2-\epsilon}$ for any $\epsilon>0$, hence we may neglect effect of genus characters assuming that $D$ is sufficiently large.

Now put $\mathscr{F} = \bigcup \mathscr{F}(D)$ where $-D$ runs over all negative odd fundamental discriminant less than $-3$ and let $\mathscr{N}(\pi)=D$ if and only if $\pi \in \mathscr{F}(D)$. We confine ourselves to odd discriminants in order to simplify the computation. One may include any fundamental discriminants $D$ via following exactly the same argument with some extra care.
\subsubsection{Summation of the class numbers in arithmetic progression}
From \cite{class} and the sieving for squarefree integers we obtain
\[
\sum_{\genfrac{}{}{0pt}{}{D:\text{squarefree, } 0<D<X}{  D\equiv a(\mathrm{mod} b)}} L(1,\chi_{-D}) \sim A_{b,a}X
\]
for some constant $A_{b,a}$, provided that $4|b$ and $a \equiv 3 (\mathrm{mod } 4)$ and $(a,b)=1$. We compute $A_{b,a}$ for some specified values:
\begin{align*}
A_{4,3}&=\frac{8}{3\pi^2}\prod_{q:odd}\frac{q^3+q^2-1}{(q^2-1)(q+1)}\\
A_{4p,a}&=A_{4,3}\frac{p^2(p+1)}{(p^3+p^2-1)(p-1)} \tag{if $\big(\frac{-a}{p}\big) =1$}\\
&=A_{4,3}\frac{p^2}{p^3+p^2-1} \tag{if $\big(\frac{-a}{p}\big) =-1$}
\end{align*}
where $p$ is an odd prime. Using the summation by parts and the identity
\[
L(1,\chi_{-D}) = \frac{\pi h_{-D}}{\sqrt{D}}
\]
we conclude:
\begin{lemma}\label{density}
Let
\begin{align*}
\tilde{\mathscr{F}}(p,\pm1)&=\bigcup_{\big(\frac{-D}{p}\big) =\pm1} \mathscr{F}(D)\\
\tilde{\mathscr{F}}(p)&=\bigcup_{p|D} \mathscr{F}(D).
\end{align*}
Then $\tilde{\mathscr{F}}(p,1)$, $\tilde{\mathscr{F}}(p,-1)$, and $\tilde{\mathscr{F}}(p)$ have asymptotic density $\frac{(p+1)p^2}{2(p^3+p^2-1)}$, $\frac{(p-1)p^2}{2(p^3+p^2-1)}$, and $\frac{p^2-1}{p^3+p^2-1}$, respectively.
\end{lemma}
\subsubsection{Distribution of the coefficients}
Let $T(D)$ be the set of characters of the ideal class group of the imaginary quadratic field $\mathbb{Q}(\sqrt{-D})$. Then, since $T(D)$ is a group, we have
\[
\sum_{\psi \in T(D)} \psi(\mathfrak{J}) = h_{-D} \omega(\mathfrak{J}),
\]
where $\omega(\mathfrak{J})=1$ if $\mathfrak{J}$ is a principal ideal and $0$ otherwise. For $\pi \in \mathscr{F}(D)$ corresponding to $\psi \in T(D)$, 
\[
\lambda_{\pi}(m) = \sum_{N(\mathfrak{J})=m} \psi(\mathfrak{J}).
\]
Therefore
\[
\sum_{\pi \in \mathscr{F}(D)} \lambda_{\pi}(m) = h_{-D} \sum_{N(\mathfrak{J})=m} \omega(\mathfrak{J}) + O(d(m)d(D)).
\]
If $m$ is an integer which is not a square and $D>m$, then $x^2+Dy^2 =m$ has no integral solution. If $m$ is a square and $D>m$, then there exists exactly one solution (up to unit) for $x^2+Dy^2 =m$. Hence:
\begin{lemma}\label{trace}
Fix an integer $m>0$. Then
\[
\frac{1}{|\mathscr{F}(D)|}\sum_{\pi \in \mathscr{F}(D)} \lambda_{\pi}(m) = \chi(m)+O_\epsilon(d(D)/D^{1/2-\epsilon}) 
\]
where $\chi(m)=1$ if $m$ is a square and $0$ otherwise.
\end{lemma}
Now recall that the local $L$-function of $\pi$ is given by
\begin{align*}
L(s,\pi_p) &= (1-\frac{\lambda_\pi(p)}{p^s}+\frac{1}{p^{2s}})^{-1} \tag{if $\big(\frac{-D}{p}\big)=1$}\\
&=(1-\frac{1}{p^{2s}})^{-1} \tag{if $\big(\frac{-D}{p}\big)=-1$}\\
&=(1-\frac{\lambda_\pi(p)}{p^s})^{-1} \tag{otherwise}.
\end{align*}
Together with Lemma \ref{density} and Lemma \ref{trace}, this yields the limiting distribution of $\lambda_\pi(p)$ for any odd prime $p$:
\[
X_p = \frac{p^3+p^2}{2(p^3+p^2-1)}\frac{1}{\pi}\frac{1}{\sqrt{4-x^2}}dx +\frac{p^2-1}{2(p^3+p^2-1)} \big(\delta_{-1}+\delta_1\big).
\]
We get the asymptotic independence of $\lambda_\pi (p)$ again from Lemma \ref{trace}. Summing up, $(\mathscr{F},\mathscr{N})$ satisfies $\circled{A}$ and $\circled{B}$.
\subsection{Symmetric powers}
Let $\pi$ be a primitive holomorphic cusp form on a squarefree level $N$ with weight $k$. The local $L$-functions on finite place are given as follows\cite{sh}:
\begin{align*}
L_p(s,\pi,\mathrm{sym}^m) &= \prod_{i=0}^m \big(1-\frac{\alpha_1(\pi_p)^{m-i}\alpha_2(\pi_p)^i}{p^s} \big)^{-1} \tag{if $(p,N)=1$}\\
&=\big(1-\frac{\lambda_\pi(p)^m}{p^s})^{-1}. \tag{if $p|N$}
\end{align*}
Let the Dirichlet series of $L(s,\pi,\text{sym}^m)$ given by
\[
L(s,\pi,\text{sym}^m)= \sum_{n=1}^\infty \frac{\lambda_{\text{sym}^m \pi}(n)}{n^s}.
\]
From Section 5.1, we know that $\mathscr{F}_m$ satisfies $\circled{A}$. Also, from the following relation 
\begin{align*}
\lambda_{\text{sym}^m \pi} (p)^2&=\lambda_\pi(p^m)^2\\
&=\sum_{j=0}^m \lambda_\pi(p^{2j})
\end{align*}
for $\pi$ unramified at $p$, we have 
\begin{align*}
E(\lambda_{\text{sym}^m \pi} (p)^2) &= 1+O(\frac{1}{p})\\
E(\lambda_{\text{sym}^m \pi} (p)) &=  O(p^{-m/2}),
\end{align*}
and therefore $\circled{B}$.

\section{Result II}
In this section, we prove Theorem \ref{cor3} and \ref{cor4}.
\subsection{Two-parameter family of elliptic curves}
For each elliptic curve $E$, there exists a unique pair of integers $a$ and $b$ ($4a^3 \neq 27b^2$) such that $E$ is isomorphic to the curve $E_{a,b}$ defined by
\[
y^2 = x^3 - ax +b
\]
and that for each prime $p$, $p^{12} \nmid (a^3,b^2)$(we call such a pair of integers $(a,b)$ \textit{minimal}). We define the naive height of the elliptic curve $E_{a,b}$ by
\[
H(E_{a,b}) = \max \{ 4|a|^3, 27b^2 \}.
\]
Let $\mathscr{F}$ be the set of $E_{a,b}$ for which $(a,b)$ is minimal. We define the ordering for $\mathscr{F}$ by the height, hence $\mathscr{N}=H$. 

It is known that for every elliptic curve $E$, the (normalized) $L$-function $L(E,s)$ attached to $E$ is automorphic \cite{bcdt}. In other words, for each $E$, there exists $\pi \in S_2(\Gamma_0(N))$ for some $N$ such that $L(E,s) = L(\pi,s)$. Therefore we may treat this two-parameter family of elliptic curves $(\mathscr{F},\mathscr{N})$ as a family of automorphic forms, and we prove Theorem \ref{cor3} via verifying $\circled{A}$ and $\circled{B}$ for $(\mathscr{F},\mathscr{N})$.

Note that if two curves are isogenous, then corresponding $\pi$ is the same. Hence the set of $L$-functions corresponding to each elements of $\mathscr{F}$ is a multi-set. It might be possible that, even if the positive-definite $L$-functions consist of density $0$ set in this multi-set, when we count them without multiplicity, the density of the positive-definite $L$-functions becomes positive. However, one can check that this is not the case, using the fact that the size of the isogeny class of elliptic curves is bounded by $8$ \cite{kenku}.

\paragraph{Remark} There are several ways to order elliptic curves; for instance one may order curves by height, discriminant, or conductor. It is expected that these orderings are comparable in the sense that the average of the quantities related to curves (average number of the points over $\mathbb{F}_p$, for instance) should be the same regardless which ordering we choose. However, among these orderings, especially when dealing with the automorphic forms or the $L$-functions associated to curves, we might want to choose the conductor. Nevertheless, we do not deal with this case in this article due to technical difficulties, although we expect Theorem \ref{cor3} to hold even when curves are ordered by the conductor.
\subsubsection{Preparation}
For any integers $a$ and $b$ with $4a^3 \neq 27b^2$, we define an elliptic curve $E_{a,b}$ by the equation
\[
y^2 = x^3 -ax +b.
\]
We write $E \sim E'$ to imply that the elliptic curves $E$ and $E'$ are isomorphic. Then
\[
E_{a,b} \sim E_{c,d}
\]
if and only if there exists a rational number $t$ such that
\[
(a,b)=(ct^4, bt^6).
\]
We put for any pair of integers $a$ and $b$
\[
h(a,b)= \max \{ 4|a|^3, 27b^2 \}.
\]
For simplicity, for integers $a$ and $b$ such that $4a^3 \neq 27b^2$, we write
\[
\lambda_{a,b}(n)=\lambda_{E_{a,b}}(n)
\]
where $\lambda_{E_{a,b}}(n)$ is the $n$-th normalized Dirichlet coefficient of the standard $L$-function attached to $E_{a,b}$. When $4a^3=27b^2$, we simply put $\lambda_{a,b}(n) =0$. 

Now for each prime $p>3$, let
\[
\alpha(p,k) = \frac{1}{p^2} \sum_{a,b} \lambda_{a,b}(p)^k
\]
where $(a,b)$ runs over each equivalence class of $(a,b)$ modulo $p$ exactly once with $4a^3 \neq 27b^2$ and $p^{12} \nmid (a^3,b^2)$. Note that when $(a,b) \equiv (0,0) (\mathrm{mod} p)$, we have $\lambda_{a,b}(p)=0$. Likewise we define $\alpha(p_1 p_2 \cdots p_r , k)$ to be the average of $\lambda_{a,b}(p_1 \cdots p_r)^k$ over the equivalence classes of $(a,b)$ modulo $p_1 \cdots p_k$.
\begin{lemma}
Fix a set of primes $p_1 , p_2 , \cdots , p_r >3$. Let $\hat{\lambda}_{a,b}(p_1 \cdots p_r)=0$ if $p_i|a$ and $p_i |b$ for some $1 \leq i\leq r$ and let $\hat{\lambda}_{a,b}(p_1 \cdots p_r) = \lambda_{a,b}(p_1 \cdots p_r)$ otherwise. Then
\[
\frac{1}{X^{5/6}2^{4/3}3^{-3/2}}\sum_{h(a,b) < X} \hat{\lambda}_{a,b}(p_1 \cdots p_r)^k = \alpha(p_1,k) \cdots \alpha(p_r,k) + O(X^{-1/3})
\]
where the implied constant depends only on $p_1, \cdots, p_r$ and $k$.
\end{lemma}
\begin{proof}
Firstly, in the set of lattices determined by $h(a,b)<X$, put $p_1 \cdots p_r \times p_1 \cdots p_r$ boxes as many as possible. Then the summation of $\hat{\lambda}_{a,b}(p_1 \cdots p_r)^k$ over each box gives $(p_1 \cdots p_r)^2\alpha(p_1 p_2 \cdots p_r, k)$ from the definition. Now, by the Chinese remainder theorem, we get
\[
\alpha(p_1 p_2 \cdots p_r, k) = \alpha(p_1,k) \cdots \alpha(p_r,k).
\]
\end{proof}
Using this lemma and considering contributions from the non-minimal pairs $(a,b)$, one proves:
\begin{lemma}\label{lemma62}
\begin{align*}
\frac{1}{X^{5/6}2^{4/3}3^{-3/2}}\sum_{\genfrac{}{}{0pt}{}{h(E_{a,b}) < X}{ 4a^3 \neq 27b^2}} &\lambda_{a,b}(p_1 \cdots p_r)^k \\
&= \prod_{j=1}^r (1-\frac{1}{p_j^{12}})^{-1} \alpha(p_j,k)  + O(X^{-1/3})
\end{align*}
\end{lemma}
\subsubsection{Distribution of $\lambda_{a,b} (p)$}
Let $S$ be the set of minimal pairs in $\mathbb{Z}^2$.
\begin{lemma}
\[
\sum_{\genfrac{}{}{0pt}{}{(a,b) \in S  }{h(a,b) < X}} = \sum_{n=1}^\infty \mu(n)\sum_{\genfrac{}{}{0pt}{}{h(a,b)< X/n^{12}}{ 4a^3 \neq 27b^2}}
\]
\end{lemma}
\begin{proof}
Let $n*$ act on $(a,b) \in \mathbb{Z}^2$ by
\[
n*(a,b)=(n^4a,n^6b)
\]
Then by the minimality of $S$,
\[
\mathbb{Z}^2 = \bigcup_{n=1}^\infty n*S \bigcup \{(a,b)| 4a^3=b^2 \}
\]
where each set is disjoint. Therefore we have
\[
\sum_{\genfrac{}{}{0pt}{}{h(a,b)< X  }{4a^3 \neq 27b^2}} = \sum_{n=1}^\infty \sum_{\genfrac{}{}{0pt}{}{(a,b) \in S }{ h(a,b) < X/n^{12}}}.
\]
Applying Mobius inversion, we get the assertion.
\end{proof}
Now by this lemma and Lemma \ref{lemma62}, we obtain:
\begin{lemma}
For any fixed set of primes $p_1 , p_2 , \cdots , p_r >3$,
\begin{align*}
E_{(\mathscr{F},\mathscr{N})}(\lambda_{a,b}(p_1 \cdots p_r)^k)&= \prod_{j=1}^r (1-\frac{1}{p_j^{12}})^{-1} \alpha(p_j,k)\\
&=\prod_{j=1}^rE_{(\mathscr{F},\mathscr{N})}(\lambda_{a,b}(p_j)^k).
\end{align*}
\end{lemma}
Therefore we conclude that the family satisfies $\circled{A}$ and $\circled{B}$, using the fact that (\cite{birch}):
\begin{align*}
\alpha(p,1) &=0\\
\alpha(p,2) &=1+O(\frac{1}{p}).
\end{align*}
\subsection{One-parameter family of elliptic curves}
Using similar ideas from the previous section, one can verify $\circled{A}$ for the family using the periodicity of $a(t)$ and $b(t)$ modulo $m\in \mathbb{N}$ and the Chinese remainder theorem. Now $\circled{B}$ follows from the following theorem for the family having non-constant $j$-invariant, with the Hecke relation $\lambda_\pi(p)^2 =1+\lambda_\pi(p^2)$.
\begin{lemma}[Katz, 1990 \cite{ka}]
There exists a constant $C>1$ depending only on the family that
\[
|E(\lambda_\pi(p^k))| \ll p^{-k/2}.
\]
\end{lemma}

\newpage
\bibliography{list}
\bibliographystyle{alpha}
\end{document}